\documentclass{amsart}
\usepackage{amsmath,amsthm,amssymb,array}

\newcommand{\Cb}{\mathbb C}

\newcommand{\Rb}{\mathbb R}

\newcommand{\Cpz}{\mathbb C \left[z\right]}

\newcommand{\Apr}{\mathcal A}
\newcommand{\Cpzn}[1]{\mathbb C^{(#1)}[z]}

\newcommand{\pin}{\pi^{(n)}}
\newcommand{\Lp}{\mathcal{L}\mathcal{P}}

\theoremstyle{plain}
\newtheorem{mythm}{Theorem}[section] 
\newtheorem{mylm}[mythm]{Lemma}

\newtheorem{myprop}[mythm]{Proposition}

\theoremstyle{definition}
\newtheorem{mydef}{Definition}[section]

\theoremstyle{remark}
 
\newtheorem*{myrem}{Remark} 

\begin{document}
\title[Linear operators preserving roots in open circular 
domains]{Linear operators on polynomials 
\\ preserving roots in open circular domains}
\date{}
\author{Eugeny Melamud}

\address{Department of Mathematics and Mechanics, 
St. Petersburg State University, 
28, Universitetskii pr., St. Petersburg, 
198504, Russia}

\email{eugeny.melamud@comapping.com}

\thanks{Linear operators on polynomial spaces, zeros of polynomials, circular domains, 
P\'olya--Schur theorem}


\dedicatory{}


\maketitle
\begin{abstract}
In the present paper we answer a question raised by J. Borcea and P. Br\"and\'en 
and give a description of the class of operators preserving roots in open 
circular domains, i.e., in images of the open upper half-plane under 
the M\"obius transformations. Our second result 
is a description of the difference between  $\Apr(G)$ 
(the class of operators preserving roots in an open set 
$G$) and $\Apr(\overline G)$ (the class of operators 
preserving roots in $\overline{G}$). 
\end{abstract}

\section{Introduction}

Let $K$ be an arbitrary set in the complex plane, by $\pi(K)$ we  
denote the set of all polynomials with roots in $K$ and by $\pi_n(K)$ we  
denote the set of polynomials from $\pi(K)$ of degree at most $n$. 
As usual, we denote by 
$\Cpz$ the set of all polynomials with complex coefficients. 
We say that a linear operator $T$ belongs to the class $\Apr(K)$, if for any $p \in \pi(K)$ the polynomial $Tp$ belongs to $\pi(K) \cup \left\{0\right\}$. This paper is concerned with the problem of description of the class $\Apr(K)$. 

Let us point out that we do not assume operators from $\Apr(K)$ to be continuous. We will need continuity in several places, but in each case operator will act on a finite-dimensional subspace of $\Cpz$, hence $T$ will be continuous on that subspace. 

Despite of a long history of the problem, it is completely 
solved only for closed circular domains (half-plane and disk) 
and their boundaries (line and circle). Probably the first widely 
known result concerning linear operators and zero distribution, is the following P\'olya--Schur theorem, 
describing the so called multiplier sequences \cite{multiplier}.

\begin{mythm} $($P\'olya--Schur, \cite{multiplier}$)$
Let $T$ be a diagonal operator corresponding to a real sequence 
$\left\{a_k\right\}_{k=0}^\infty$ given by $T[z^k] = a_k z^k$. 
Then the following are equivalent:
\begin{enumerate}
\item The operator $T$ preserves polynomials with real zeros;
\item For any positive integer $n$ the polynomials 
$$T\left[(1+z)^n\right]$$ 
belong to the Laguerre--P\'olya class $\Lp(\mathbb{R}_+) \cup \Lp(\mathbb{R}_-)$; 
\item The series
$$T[e^z] = \sum_{k=0}^\infty \frac{c_k}{n!} z^k$$
converges in the whole plane and belongs 
to the class $\Lp(\mathbb{R}_+) \cup \Lp(\mathbb{R}_-)$.
\end{enumerate}
\end{mythm}

The mentioned classes $\Lp(\mathbb{R}_+) $ and 
$\Lp(\mathbb{R}_-)$ play an important role in the 
theory of entire functions; they consist of entire functions 
which are limits, uniformly on compact sets, of polynomials with 
zeros respectively in $\mathbb{R}_+ = [0,\infty)$ or in 
$\mathbb{R}_- = (-\infty, 0]$. These 
classes are a key object in various classifications 
of operators preserving roots on the real axis.

Later a similar result was obtained for operators commuting with differentiation operator $D$ \cite{benz} and for operators commuting with inverted differentiation $z^2 D$ \cite{haakan}. Still the description of operators which preserve roots on the real axis or in the half-plane and do not satisfy some special conditions was an open problem till recently. A classification of linear operators preserving roots in circular domains and on their boundaries 
obtained recently by J. Borcea and P. Br\"and\'en \cite{borcea_branden_arxiv} was a real breakthrough. To introduce their result we will need the following definition. 
\begin{mydef}
Let $K$ be an arbitrary set in the complex plane. A polynomial $p \in \Cb[z_1,\ldots, z_n]$ is said to be $K$-stable, if $p(\zeta_1,\ldots,\zeta_n) \neq 0$ whenever $\zeta_i \in K$, 
$1 \leq i \leq n$. 
\end{mydef} 

Denote by $\Phi(z)$ the M\"obius transformation given by 
$$\Phi(z) = \frac{az + b}{cz + d}, \quad \text{ where } \quad ad - bc \neq 0.$$

For a linear operator $T: \Cpz \to \Cpz$, let us extend it to the set of bivariate polynomials $\Cb [z, w]$ as if $w$ is a constant, by putting $T[z^k w^m] = w^m T[z^k]$.

Denote by $H$ the  upper half-plane $\left\{z: \Im z > 0\right\}$ and put  
$C = \Phi^{-1}(H)$. The following theorem \cite{borcea_branden_arxiv} describes the class $\Apr(C')$, where $C' = 
\Cb \setminus C$ stands for the complement of $C$.

\begin{mythm} 
$($\cite{borcea_branden_arxiv}$)$
\label{closed1}
Let $T:\Cpz \rightarrow \Cpz$ be a linear operator. Then $T \in \Apr(C')$ if and only if, either
\begin{enumerate}
\item The range of $T$ is of dimension $1$ and $T$ can be written as
$$T(f) = \alpha(f)P,$$
where $\alpha$ is a linear functional, and $P \in \pi(C^{'})$, or
\item For any non-negative integer $n$ the polynomial
$$T\left[\left((az + b)(cw + d) + (aw + b)(cz + d)\right)^{n}\right]$$
is $C$-stable.
\end{enumerate}
\end{mythm}
There is a similar theorem for operators acting on $\Cb_n[z]$, the space of all polynomial of degree at most $n$. 
\begin{mythm}
$($\cite{borcea_branden_arxiv}$)$
\label{closed2}
Let $T:\Cb_n[z] \rightarrow \Cpz$ be a linear operator. Then $T$ 
acts from $\pi_n(C')$ to $\pi(C') \cup \left\{0\right\}$ if and only if, either
\begin{enumerate}
\item The range of $T$ is of dimension $1$ and $T$ can be written as
$$T(f) = \alpha(f)P,$$
where $\alpha$ is a linear functional, and $P \in \pi(C^{'})$,
or
\item The polynomial
$$T\left[\left((az + b)(cw + d) + (aw + b)(cz + d)\right)^{n}\right]$$
is $C$-stable.
\end{enumerate}
\end{mythm}

In \cite{borcea_branden_arxiv} a similar description is also given for 
the class $\Apr(K)$ in the case where $K$ is a line or a circle. 
However, in the case where $K$ is an {\it open circular domain} 
the description of $\Apr (K)$ was formulated as an open 
problem in \cite{borcea_branden_arxiv, workshop} as well 
as in \cite[Section 5]{plms}, where conditions which are 
separately necessary or sufficient were obtained for an 
operator to preserve roots in the open half-plane.

The aim of this note is to give a complete description of 
operators preserving roots in an open circular domain which 
is similar to Theorem \ref{closed1}. It will be derived from 
Theorem \ref{closed1} and from Theorem \ref{difference}, describing 
the difference between the classes $\Apr(G)$ and $\Apr(\overline{G})$, 
where $G$ is an open set in $\Cb$. 
Theorem \ref{difference} is, in its turn, a corollary of the classical Hurwitz theorem.

\section{Difference between $\Apr(G)$ and $\Apr(\overline{G})$}

In this section $G$ is always an open set in the complex plane. First, let us show that $\Apr(G) \subset \Apr(\overline{G})$.

\begin{myprop}
Let $T: \Cpz \to \Cpz$ be a linear operator. If $T \in \Apr(G)$, then $T \in \Apr(\overline{G})$.
\end{myprop}
\begin{proof}
Assuming the opposite, we can find a polynomial $p \in \pi(\overline{G})$ such that $Tp \notin \pi(\overline{G})$. Then there exists a root $z_0$ of $p(z)$ lying on $\partial G$ (otherwise $Tp \in \pi(G)$). Let $p(z) = (z - z_0) r(z)$. Consider a family of polynomials:
$$p_{\zeta}(z) = p(z) - \zeta r(z) = (z - z_0 - \zeta) r(z).$$
By continuity argument, there exists 
$\zeta$ such that $z_0 + \zeta \in G$ and $Tp_\zeta \notin \pi(\overline{G})$. Applying the similar reasoning for each root of $p(z)$, which belongs to $\partial G$, we will get $q \in \pi(G)$ such that $Tq \notin \pi(G)$. We come to a contradiction with the condition $T \in \Apr(G)$. 
\end{proof}

However, $\Apr(G) \neq \Apr(\overline{G})$. It is easy to find an operator from $\Apr(\overline{G}) \setminus \Apr(G)$ (e.g., let $G$ be an open upper half-plane and let $(Tp)(z) = z p(z)$). 

Denote by $\Cpzn{n}$ the space of polynomials of degree exactly $n$, and denote by  $\pin(K)$ the set of polynomials of degree exactly $n$ with all roots in the set $K$. For $\mathcal{P} \subset \Cpz$, denote by $GCD(\mathcal{P})$ the greatest common divisor of all polynomials from $\mathcal{P}$.

\begin{mylm}
\label{second}
Let $G \subset \Cb$ be an open set and $T \in \Apr(\overline{G})$. Then the following are equivalent:
\begin{enumerate}
\item There exists $p \in \pin(G)$ such that $Tp \notin \pi(G) \cup \{0\}$;
\item $GCD(T(\pin(G)))$ has a root on $\partial G$.
\end{enumerate}
\end{mylm}
\begin{proof}
Implication $2 \Longrightarrow 1$ is obvious. Let us prove $1 \Longrightarrow 2$. Assuming the opposite, we can find two polynomials $p, q \in \pin(G)$ and  $\xi \in \partial G$ such that $(Tp)(\xi) = 0$, but $(Tq)(\xi) \neq 0$. Without loss of generality we may assume that $0 \in \partial G$ and $\xi = 0$. Let $m$ be a multiplicity of zero as a root of $T(p)$. Therefore $T(p)(z) = z^m R(z)$. If we put  $Q = Tq$, then $Q(0) \ne 0$. Since $G$ is open, applying the Hurwitz theorem we can find $\varepsilon > 0$ such that $p + \zeta q \in \pin(G)$ for all $\zeta$
with $|\zeta|<\varepsilon$. Thus, $T(p + \zeta q) = z^m R(z) + \zeta Q(z) \in \pi(\overline{G})$. Choose 
$z \notin \overline G$ with $|z|$ small enough so that
$$\left|\zeta\right| = \left|-z^m \frac{R(z)}{Q(z)}\right| < \varepsilon$$
Then $T(p + \zeta q)(z) = 0$, which contradicts the assumption 
that $T(p + \zeta q) \in \pi(\overline G)$.
\end{proof}
Actually, one can replace the second condition in the above lemma with a condition, which is much easier to check. We need a simple statement for that.
\begin{mylm}
\label{third}
Let $T: \pi_n (G) \to \pi (G) \cup \{0\}$ be a linear operator. If $n > k > m$, then $$GCD(T(\pi^{(m)} (G))) \ \vdots \ GCD(T(\pi^{(k)} (G))).$$
\end{mylm}
\begin{proof} Without loss of generality we may take $m = k - 1$. Let $a \in G$, one can find $\epsilon \in \Cb$ such that $a + \epsilon \in G$. For arbitrary $p \in \pi^{(k - 1)} (G)$ we have
$$T\left[(z-a-\epsilon)p(z)\right](\zeta) = \left[T\big((z - a)p\big)\right](\zeta) - \epsilon (Tp)(\zeta).$$
Denote by $\zeta_0$ a root of $GCD(T(\pi^{(k)} (G)))$. If we put $z = \zeta_0$ in the above equality, we get
$$(Tp)(\zeta_0) = 0.$$
So, every root of $GCD(T(\pi^{(k)} (G)))$ will also 
be a root of any polynomial from $T(\pi^{(k - 1)} (G))$ 
of at least the same multiplicity and, therefore, 
$$GCD(T(\pi^{(k - 1)} (G))) \ \vdots \ GCD(T(\pi^{(k)} (G))).$$
\end{proof}
Now we are ready to describe the difference between $\Apr(G)$ and $\Apr(\overline{G})$.
\begin{mythm}
\label{difference}
Let $G \subset \Cb$ be an open set. Consider a linear operator $T: \Cpz \to \Cpz$ preserving roots in $\overline{G}$. The following are equivalent:
\begin{enumerate}
\item The operator $T$ belongs to $\Apr(G)$;
\item If $k$ is a minimal non-negative integer such that $T[z^k]$ is not identically zero,
then $T[z^k]$ has no roots on $\partial G$.
\end{enumerate} 
\end{mythm}
\begin{proof} The second statement trivially follows from 
the first one. Let us prove that 
$1 \Longrightarrow 2$. Assuming the converse, 
we can find a polynomial $p \in \pi(G)$ such that $Tp \notin \pi(G)\cup\{0\}$. 
Let $n$ be the degree of $p$. Applying Lemma \ref{second}, we  get that 
$GCD(T(\pin(G)))$ has a root on $\partial G$. 
Hence, by Lemma \ref{third}, 
$GCD(T(\pi^{(k)}(G))) = c T[z^k]$ has a root on $\partial G$. 
We arrive at a contradiction. 
\end{proof}
\begin{myrem} Let $T: \Cpz \to \Cpz$ be a linear operator. Consider its expansion as a series with respect to the differentiation operator $D$:
$$T = \sum_{k=0}^\infty Q_k D^k.$$
It is easy to see, that the second condition in Theorem \ref{difference} is equivalent to $Q_k(z) \neq 0$ for all $z \in \partial G$, where $Q_k$ is the first non-zero polynomial coefficient in the above expansion. 
\end{myrem}

In conclusion we give an example of an operator 
$T \in \Apr (\overline{G}) \setminus \Apr(G)$ 
such that $GCD(\Cb_{n}[z])$ has roots on $\partial G$ 
for some non-negative integer $n$, while $GCD(\Cb_{n+1}[z])$ does not. 
Let $G$ be the upper half-plane $H$. Consider an operator $T$ given by 
$$
(Tp)(z) = a_{n+1} - a_0 z,
$$
where $p(z) = \sum_{j=0}^\infty a_j z^j $ 
and only finitely many coefficients $a_j$ 
are non-zero. If all roots of $p$ are real, then the same is true 
for $Tp$. So $T \in \Apr(\Rb)$. It is known 
\cite[Section 3.1, Lemma 3]{borcea_branden_arxiv} that 
if $T \in \Apr(\Rb)$, then it maps $\pi(\overline{H})$ 
either to $\pi(\overline{H})$ or to $\pi(-\overline{H})$. 
It is easy to see that $T\left(\pi(\overline{H})\right) 
\subset \pi(\overline{H})$, so $T \in \Apr (\overline{G})$. 
Finally, $GCD(\Cb_{n}[z]) = z$ has a root on $\Rb$, while 
$GCD(\Cb_{n+1}[z]) = 1$ does not.
 
\section{Classification of linear operators preserving roots in open circular domain}
In this section we will derive 
from Theorems \ref{closed1} and \ref{difference}
a description of linear 
operators preserving roots in open circular domain. 
As in Section 1, $\Phi(z) = \frac{az + b}{cz + d}$ and $C = \Phi^{-1}(H)$, where $H$ is the open upper half-plane $\left\{z: \Im z > 0\right\}$. Denote by $C^r$ the set $Int(C')$, that is, the interior of the complement $ C'$ of $C$. 

For the sake of brevity, put
$$ Q_n(z, w)  = \big((az + b)(cw + d) + (aw + b)(cz + d)\big)^{n}.$$
The next theorem is an analogue of Theorem \ref{closed1} for open circular domains.

\begin{mythm} 
\label{mainth}
Let $T:\mathbb{C}[z] \rightarrow \mathbb{C}[z]$ be a linear operator. Then $T \in \Apr (C^r)$ if and only if either
\begin{enumerate}
\item The range of $T$ is of dimension $1$ and $T$ can be written as
$$T(f) = \alpha(f)P,$$
where $\alpha$ is a linear functional and $P \in \pi(C^{'})$, or

\item For any non-negative integer $n$ such that the polynomial $T[Q_n(z, w)]$ 
is not identically zero, 
$(TQ_n)(z, w) \neq 0$ whenever $z \in \overline{C}$ and $w \in C$.
\end{enumerate}
\end{mythm}
\begin{proof}
Case (1) is trivial, so let us assume, that the range of $T$ 
has dimension greater than $1$. First we prove the necessity of (2). Consider the equation
$$(az + b)(cw + d) + (aw + b)(cz + d) = 0,$$
which is equivalent to 
$$\frac{az + b}{cz + d} = - \frac{aw + b}{cw + d} \ .$$
If $w \in C$ and $z \in \overline{C}$, then 
$$\frac{az + b}{cz + d} = \Phi(z) \in \overline{H},$$
and, therefore,
$$- \frac{aw + b}{cw + d} \in \overline{H}.$$
Thus, $w \in C'$, but this contradicts the assumption $w \in C$. 
So we have shown that $Q_n(z, w) \neq 0$ if $z \in \overline{C}$ and $w \in C$. 
Hence, for any fixed $w_0 \in C$, the univariate polynomials $Q_n(z, w_0)$ 
belong to $\pi(C^r)$. Since $T \in \Apr(C^r)$, the same is true for 
the polynomials $T[Q_n(z, w_0)]$. Therefore $(TQ_n)(z, w) \neq 0$, 
if $z \in \overline{C}$ and $w \in C$.

We turn to the proof of sufficiency. It follows from condition (2) 
that the polynomials $T[Q_n(z, w)]$ are $C$-stable. Hence, by 
Theorem \ref{closed1}, $T \in \Apr(C')$. Let $k$ be the minimal non-negative integer such that $T[z^k] \neq 0$. By Theorem \ref{difference}, if $T \notin \Apr(C^r)$ then $T[z^k]$ should have roots on $\partial C^r = \partial C $. Therefore, $(TQ_k) (z, w) = \alpha(w) T[z^k]$ must vanish for some $z \in \overline{C}$ and $w \in C$. We arrive at a contradiction.
\end{proof}

\begin{myrem}
Condition (2) in Theorem \ref{mainth} may be replaced by a slightly 
weaker condition. Let $k$ be the minimal non-negative integer such that 
$T[z^k]$ is not identically zero. Then condition (2) holds if and only if 
all polynomials $T[Q_n(z, w)]$, $n > 0$, are $C$-stable, 
and $(TQ_k)(z, w) \neq 0$ for $z \in \overline{C}$ 
and $w \in C$. This can be seen by applying the same reasoning as in 
the proof of Theorem \ref{difference}.
\end{myrem}

However, one can not obtain similar classification of 
operators which act on the space $\Cb_n[z]$ and preserve 
roots in an open circular domain. The following statement 
could be an analogue of Theorem \ref{closed2} for open circular domains: 
\medskip

"{\it Let $T:\mathbb{C}_n[z] \rightarrow \mathbb{C}[z]$ be a linear operator. Then $T(\pi_n(C')) \subset \pi(C')$ 
if and only if either
\begin{enumerate}
\item The range of $T$ is of dimension $1$ and $T$ can be written as:
$$T(f) = \alpha(f)P,$$
where $\alpha$ is a linear functional and $P \in \pi(C^{'})$, or

\item $(TQ_n)(z, w) \neq 0$ whenever $z \in \overline{C}$ and $w \in C$."
\end{enumerate} }
\medskip

We show that this statement is not correct; 
namely, condition (2) turns out to be not sufficient. 
Consider the  operator $(Tp)(z) = p'(z) - zp(z)$. Let $L$ 
be the lower half-plane $\{z: \Im z < 0\}$, in this case 
$Q_n(z, w)=(z + w)^n$. First note that  $T[1] = z$ and 
so $T \notin \Apr (L)$. 
However, by Theorem \ref{closed1} and below reasoning 
$T \in \Apr (\overline{L})$. Let us evaluate $T\left[(z + w)^n\right]$:
$$
T\left[(z + w)^n\right] = n (z + w)^{n-1} - z(z + w)^n = (z + w)^{n-1} 
\left[n - z(z+w)\right].
$$
Suppose that $n > 0$. Let us see that the obtained polynomials do not 
vanish when $z \in \overline{H}$ 
and $w \in H$. Since $(z + w)^{n-1} \neq 0$ for 
$z \in \overline{H}$ and $w \in H$, we look at the polynomial 
$q(z,w) = n - z(z+w)$. For $z \in \overline{H}$ and $w \in H$, 
we have $0 \leq \arg z \leq \pi$, $0 < \arg w < \pi$, 
whence $0 \leq \arg (z + w) \leq \pi$ and  $0 < \arg[z(z + w)] < 2 \pi$. 
Clearly, in this case $q(z, w) \ne 0$. So, if $n > 0$, then 
$T\left[(z + w)^n\right] \neq 0$ for $z \in \overline{H}$ and 
$w \in H$. Thus, condition (2) is fulfilled but it is not sufficient
for the inclusion $T(\pi_n(L)) \subset \pi(L)$.

\section{Acknowledgements}
The author wishes to thank Anton Baranov for numerous advises 
and discussions. The author is also grateful to Petter Br\"and\'en 
for useful remarks, especially for pointing out an inaccuracy in the 
initial version of Lemma \ref{second}.

\end{document}